\documentclass[11pt]{amsart}
\usepackage[utf8]{inputenc}
\usepackage{enumerate,amssymb,xcolor,comment}
\usepackage[a4paper, footskip=1cm, headheight = 16pt, top=3.7cm, bottom=3cm,  right=2.5cm,  left=2.5cm ]{geometry}
\usepackage[colorlinks,linkcolor=blue,citecolor=red]{hyperref}
\usepackage{thm-restate}

\newtheorem{question}{Question}[section]
\newtheorem{lemma}[question]{Lemma}
\newtheorem{theorem}[question]{Theorem}
\newtheorem{conjecture}[question]{Conjecture}

\usepackage[foot]{amsaddr}

\usepackage[T1]{fontenc}

\usepackage[leqno]{amsmath}

\makeatletter
\newcommand{\leqnomode}{\tagsleft@true}
\newcommand{\reqnomode}{\tagsleft@false}
\makeatother

\usepackage{latexsym}
\usepackage{amsfonts}

\usepackage{tikz}
\usepackage{float}
\usepackage{lmodern}

\def\dd{\hbox{-}}

\usepackage{marvosym}               

\DeclareMathOperator{\tw}{tw}

\DeclareMathOperator{\Hub}{Hub}

\newcounter{tbox}

\newcommand{\sta}[1]{\vspace*{0.3cm}\refstepcounter{tbox}\noindent{ \parbox{\textwidth}{(\thetbox) \emph{#1}}}\vspace*{0.3cm}}

\newcommand{\mylongtitle}[1]{%
  \ifodd\value{page}%
    \protect\parbox{0.97\linewidth}{#1}\hfill%
  \else%
    \hfill\protect\parbox{0.97\linewidth}{#1}%
  \fi%
}

\renewcommand{\S}{\mathcal{S}}

\makeatletter
\newcommand{\otherlabel}[2]{\protected@edef\@currentlabel{#2}\label{#1}}
\makeatother

\mathchardef\mh="2D

\title[Induced subgraphs and tree decompositions IV.]{Induced subgraphs and tree decompositions IV.\\(Even hole, diamond, pyramid)-free graphs}

\author{Tara Abrishami$^{\ast \dagger}$}
\author{Maria Chudnovsky$^{\ast \dagger}$}
\author{Sepehr Hajebi $^{\mathsection}$}
\author{Sophie Spirkl$^{\mathsection \parallel}$}
\address{$^{\ast}$Princeton University, Princeton, NJ, USA}
\address{$^{\mathsection}$Department of Combinatorics and Optimization, University of Waterloo, Waterloo, Ontario, Canada}
\address{$^{\dagger}$ Supported by NSF Grant DMS-1763817 and
     NSF-EPSRC Grant DMS-2120644.}
\address{$^{\parallel}$ We acknowledge the support of the Natural Sciences and Engineering Research Council of Canada (NSERC), [funding reference number RGPIN-2020-03912].
Cette recherche a \'et\'e financ\'ee par le Conseil de recherches en sciences naturelles et en g\'enie du Canada (CRSNG), [num\'ero de r\'ef\'erence RGPIN-2020-03912].}

\date {\today}

\allowdisplaybreaks

\begin{document}

\maketitle
\begin{abstract}
A \textit{hole} in a graph $G$ is an induced cycle of length at least four, and an \textit{even hole} is a hole of even length. The \textit{diamond} is the graph obtained from the complete graph $K_4$ by removing an edge. A \textit{pyramid} is a graph consisting of a triangle called the \textit{base}, a vertex called the \textit{apex}, and three internally disjoint paths starting at the apex and disjoint otherwise, each joining the apex to a vertex of the base. For a family $\mathcal{H}$ of graphs, we say a graph $G$ is $\mathcal{H}$-\textit{free} if no induced subgraph of $G$ is isomorphic to a member of $\mathcal{H}$. Cameron, da Silva, Huang, and Vu\v{s}kovi\'{c} proved that (even hole, triangle)-free graphs have treewidth at most five, which motivates studying the treewidth of even-hole-free graphs of larger clique number. Sintiari and Trotignon provided a construction of (even hole, pyramid, $K_4$)-free graphs of arbitrarily large treewidth. 
 
Here, we show that for every $t$, (even hole, pyramid, diamond, $K_t$)-free graphs have bounded treewidth. The graphs constructed by Sintiari and Trotignon contain diamonds, so our result is sharp in the sense that it is false if we do not exclude diamonds. Our main result is in fact more general, that treewidth is bounded in graphs excluding certain wheels and three-path-configurations, diamonds, and a fixed complete graph. The proof uses “non-crossing decompositions” methods similar to those in previous papers in this series. In previous papers, however, bounded degree was a necessary condition to prove bounded treewidth. The result of this paper is the first to use the method of “non-crossing decompositions” to prove bounded treewidth in a graph class of unbounded maximum degree.\end{abstract}
\section{Introduction}
All graphs in this paper are simple. Let $G = (V(G), E(G))$ be a graph. An {\em induced subgraph} of $G$ is a subgraph of $G$ formed by deleting vertices. In this paper, we use induced subgraphs and their vertex sets interchangeably. A {\em tree decomposition $(T, \chi)$ of $G$} consists of a tree $T$ and a map $\chi: V(T) \to 2^{V(G)}$ satisfying the following: 
\begin{enumerate}[(i)]
\item For all $v \in V(G)$, there exists $t \in V(T)$ such that $v \in \chi(t)$; 

\item For all $v_1v_2 \in E(G)$, there exists $t \in V(T)$ such that $v_1, v_2 \in \chi(t)$; 
\item For all $v \in V(G)$, the subgraph of $T$ induced by $\{t \in V(T) \text{ s.t. } v \in \chi(t)\}$ is connected. 
\end{enumerate}

The {\em width} of a tree decomposition $(T, \chi)$ of $G$ is $\max_{t \in V(T)} |\chi(t)| - 1$. The {\em treewidth} of $G$, denoted $\tw(G)$, is the minimum width of a tree decomposition of $G$. Treewidth was introduced by Robertson and Seymour in their study of graph minor theory. 

Treewidth is roughly a measure of how ``complicated'' a graph is: forests have treewidth one, and in general, the smaller the treewidth, the more ``tree-like'' (and thus ``uncomplicated'') the graph. Graphs with bounded treewidth have nice structural properties, and many classic NP-hard problems can be solved in polynomial time in graphs with bounded treewidth (see \cite{dynamic-programming} for more details). Understanding which graphs have bounded treewidth is an important question in the field of structural graph theory. This question is usually explored by considering substructures of graphs that, when present, cause the treewidth to be large, and when absent, guarantee that the treewidth is small. Robertson and Seymour famously gave a complete answer to this question in the case of subgraphs. By $W_{k \times k}$ we denote the {\em $(k \times k)$-wall}; see \cite{wallpaper} for a full definition. 
\begin{theorem}[\cite{RS-GMV}]
There is a function $f : \mathbb{N} \to \mathbb{N}$ such that every graph of treewidth at least
$f(k)$ contains a subdivision of $W_{k \times k}$ as a subgraph. 
\end{theorem}

Recently, the study of which graphs have bounded treewidth has focused on {\em hereditary} graph classes; that is, classes of graphs defined by forbidden induced subgraphs. Most conjectures and theorems about the treewidth of hereditary graph classes fall into one of two categories: bounded treewidth in graph classes with bounded maximum degree, and logarithmic treewidth in graph classes with arbitrary maximum degree and bounded clique number. The main open question in the former category was the following: 
\begin{conjecture}[\cite{aboulker}]\label{conj:wallconj}
For all $k, \Delta > 0$, there exists $c = c(k, \Delta)$ such that every graph with
maximum degree at most $\Delta$ and treewidth more than $c$ contains a subdivision of $W_{k\times k}$ or the
line graph of a subdivision of $W_{k \times k}$ as an induced subgraph.
\end{conjecture}

In earlier papers in this series, several special cases of Conjecture \ref{conj:wallconj} were resolved (see \cite{wallpaper}, \cite{ehf}). Conjecture \ref{conj:wallconj} was recently proven in \cite{Korhonen}, by a different method.
In this paper, we use techniques similar to  \cite{wallpaper} and \cite{ehf}
to prove that a hereditary graph class of arbitrarily large maximum degree has bounded treewidth. 

Before we state our main result, we define several types of graphs.
If $G$ is a path or a cycle, the {\em length} of $G$ is $|E(G)|$. By $P_k$ we denote the path on $k$ vertices. If $P = p_1 \dd p_2 \dd \cdots \dd p_k$, then $P^* = P \setminus \{p_1, p_k\}$ denotes the {\em interior} of the path $P$. A {\em hole} of $G$ is an induced cycle of length at least four. By $C_4$ we denote the hole of length four. Let $G$ be a graph and let $v \in V(G)$. The {\em open neighborhood of $v$ in $G$}, denoted $N_G(v)$, is the set of vertices of $V(G)$ adjacent to $v$. The {\em closed neighborhood of $v$ in $G$}, denoted $N_G[v]$, is the union of $v$ and $N_G(v)$. Let $X \subseteq V(G)$. The {\em open neighborhood of $X$ in $G$}, denoted $N_G(X)$, is the set of vertices of $V(G) \setminus X$ with a neighbor in $X$. The {\em closed neighborhood of $X$ in $G$}, denoted $N_G[X]$, is the union of $X$ and $N_G(X)$. When the graph $G$ is clear, we omit the subscript $G$ from the open and closed neighborhoods. A {\em wheel} $(H, w)$ is a hole $H$ and a vertex $w \in V(G)$ such that $w$ has at least three pairwise nonadjacent neighbors in $H$. A {\em line wheel} $(H, v)$ is a hole $H$ and a vertex $v \not \in H$ such that $|H \cap N(v)|$ is the union of two disjoint edges. An {\em even wheel} is a line wheel or a wheel $(H, w)$ such that $|N(w) \cap H|$ is even and $N(w) \cap H$ is not a path of length one. If $X, Y \subseteq V(G)$, we say {\em $X$ is anticomplete to $Y$} if there are no edges with one endpoint in $X$ and one endpoint in $Y$. We say that $X$ {\em has a neighbor in $Y$} if $X$ is not anticomplete to $Y$. We say that $v$ is anticomplete to $X$ if $\{v\}$ is anticomplete to $X$. 

A {\em diamond} is the graph given by deleting an edge from $K_4$. A {\em theta} is a graph consisting of two non-adjacent vertices $a$ and $b$ and three paths $P_1$, $P_2$, $P_3$ from
$a$ to $b$ of length at least two, such that $P^*_1$, $P^*_2$, $P_3^*$ are pairwise disjoint and anticomplete to each other. We say this is a theta {\em between $a$ and $b$ through $P_1$, $P_2$, and $P_3$}. A {\em pyramid} is a graph consisting of a vertex $a$, a triangle $b_1, b_2, b_3$, and three paths $P_1$, $P_2$, $P_3$ from $a$ to $b_1, b_2, b_3$, respectively, such that $P_1 \setminus \{a\}$, $P_2 \setminus \{a\}$, and $P_3 \setminus \{a\}$ are pairwise disjoint and anticomplete to each other, and at least two of $P_1, P_2, P_3$ have length at least two. The vertex $a$ is called the {\em apex} of the pyramid. A {\em prism} is a graph consisting of two disjoint triangles $a_1a_2a_3$ and $b_1b_2b_3$ and three paths $P_1$, $P_2, P_3$, with $P_i$ from $a_i$ to $b_i$, such that for all distinct $i, j \in \{1, 2, 3\}$, the only edges between $P_i$ and $P_j$ are $a_ia_j$ and $b_ib_j$. Thetas, pyramids, and prisms are called {\em three-path configurations}. 

If $H$ is a graph, then by {\em $H$-free graphs} we mean the class of graphs which do not contain $H$ as an induced subgraph. If $\mathcal{H}$ is a set of graphs, then by {\em $\mathcal{H}$-free graphs} we mean the class of graphs which are $H$-free for every $H \in \mathcal{H}$. Let $\mathcal{C}$ be the class of ($C_4$, diamond, theta, pyramid, prism, even wheel)-free graphs, and let $\mathcal{C}_t$ be the class of ($C_4$, diamond, theta, pyramid, prism, even wheel, $K_t$)-free graphs. The main result of this paper is the following theorem.  
\begin{theorem}
\label{thm:main-nonspecific}
For all $t > 0$ there exists $c_t \geq 0$ such that $\tw(G) \leq c_t$ for every $G \in \mathcal{C}_t$. 
\end{theorem} 
Note that thetas, prisms, and even wheels contain even holes, so (even hole, diamond, pyramid, $K_t$)-free graphs are a subclass of $\mathcal{C}_t$. Therefore, Theorem \ref{thm:main-nonspecific} implies the following: 

\begin{theorem}
\label{thm:ehf-nonspecific}
For all $t > 0$ there exists $d_t \geq 0$ such that $\tw(G) \leq d_t$ for every (even hole, diamond, pyramid, $K_t$)-free graph $G$. 
\end{theorem}

Theorem \ref{thm:main-nonspecific} is the first result of this series that gives a constant bound on treewidth in a class of graphs with arbitrary maximum degree; the previous results have either obtained a constant bound on treewidth in graph classes with bounded degree (\cite{wallpaper}, \cite{ehf}), or given a logarithmic bound on treewidth in graph classes with bounded clique number (\cite{logpaper}).
 Bounded treewidth results for similar graph classes were also proved in \cite{cameron-vuskovic}.

In \cite{layered-wheels}, Sintiari and Trotignon construct (even hole, pyramid, $K_4$)-free graphs of arbitrarily large treewidth. These graphs contain diamonds; therefore, Theorem \ref{thm:ehf-nonspecific} is sharp in the sense that excluding diamonds is necessary to obtain bounded treewidth.  Sintiari and Trotignon also made the following conjecture: 

\begin{conjecture}[\cite{layered-wheels}]
\label{conj:ehf-diamond-k4}
(Even hole, diamond, $K_4$)-free graphs have bounded treewidth. 
\end{conjecture}

Theorem \ref{thm:ehf-nonspecific} is a special case of Conjecture \ref{conj:ehf-diamond-k4}. If Conjecture \ref{conj:ehf-diamond-k4} can be proven using techniques similar to those used in this paper and in the previous papers of this series, then Theorem \ref{thm:ehf-nonspecific} is the base case to prove Conjecture \ref{conj:ehf-diamond-k4}. Indeed, we conjecture the following slight generalization of Conjecture \ref{conj:ehf-diamond-k4}. Let $\mathcal{C}_t^*$ be the class of ($C_4$, diamond, theta, prism, even wheel, $K_t$)-free graphs.  
\begin{conjecture} 
For all $t > 0$ there exists $c_t \geq 0$ such that $\tw(G) \leq c_t$ for every $G \in \mathcal{C}_t^*$. 
\label{conj:generalization}
\end{conjecture}
In view of Conjecture \ref{conj:generalization}, when possible, we prove the results of this paper for graphs in $\mathcal{C}_t^*$ instead of $\mathcal{C}_t$. 
\subsection{Proof outline}
Here, we give a brief outline of the ideas used in the proof of Theorem \ref{thm:main-nonspecific}. Many of these ideas, or similar ones, appear in previous papers in this series. One major tool we use is that of balanced separators; a graph has a balanced separator for every normalized weight function on its vertices if and only if it has bounded treewidth. 

Let $G \in \mathcal{C}_t$. If $G$ is also wheel-free, then $G$ has bounded treewidth (by Theorem \ref{thm:wheel-free}). Therefore, we would like to apply decomposition techniques to $G$ to obtain an induced subgraph $\beta$ of $G$ such that: (i) the treewidth of $\beta$ is easy to compute (i.e. because $\beta$ is wheel-free), and (ii) there exists a function $f$ such that if $\tw(\beta) \leq c$, then $\tw(G) \leq f(c)$.

To obtain property (i), we make use of star cutsets associated with wheel centers. A {\em star cutset} of a graph $G$ is a set $C \subseteq V(G)$ such that $G \setminus C$ is not connected and there exists $v \in V(G)$ such that $C \subseteq N[v]$. We call $v$ the {\em center} of the star cutset $C$. Let $(H, v)$ be a wheel of $G$. Then, $v$ is the center of a star cutset $C$ of $G$, and $H$ is not contained in the closed neighborhood of any connected component of $G \setminus C$ (by Lemma \ref{lemma:forcer_lemma}). Therefore, the star cutset with center $v$ ``breaks'' the hole $H$. If $\beta$ is contained in the closed neighborhood of a connected component of $G \setminus C$, then $\beta$ does not contain the wheel $(H, v)$. Therefore, star cutsets associated with wheel centers are a promising way to construct an induced subgraph $\beta$ whose treewidth is easy to compute. 

To obtain property (ii), we make use of the relationship between treewidth and collections of decompositions with a property called ``non-crossing.'' ``Non-crossing decompositions'' interact well with treewidth, and provide a way to obtain a function $f$ such that if $\tw(\beta) \leq c$, then $\tw(G) \leq f(c)$. It turns out that there are natural decompositions corresponding to star cutsets, and in the case of graphs in $\mathcal{C}_t$, these decompositions are ``nearly non-crossing'' (a slight generalization of non-crossing that also cooperates with treewidth). Because of the way the decompositions corresponding to star cutsets are constructed, we obtain an induced subgraph $\beta$ of $G$ such that either $\beta$ is wheel-free or $\beta$ has a balanced separator. To prove that $\beta$ has a balanced separator if it is not wheel-free, we use degeneracy to bound the degree of vertices which are wheel centers in $\beta$ (a similar technique was used in \cite{logpaper}). In either case, $\beta$ has bounded treewidth, and so $G$ has bounded treewidth. 

In previous papers in this series, decompositions corresponding to star cutsets were also used to reduce the problem of bounding the treewidth of $G$ to bounding the treewidth of a ``less complicated'' induced subgraph $\beta$ of $G$. But the collections of decompositions used in previous results were not nearly non-crossing; instead, we used that the graph classes had bounded degree to partition the decompositions into a bounded number of nearly non-crossing collections. In $\mathcal{C}_t$, we were able to slightly modify the star cutsets we consider in order to obtain a single collection of non-crossing decompositions. This eliminated the need for the bounded degree condition required for previous results of this series, and allowed us to prove for the first time that the treewidth of a graph class with unbounded maximum degree is bounded. 

\subsection{Organization} 
The remainder of the paper is organized as follows. In Section \ref{sec:tools}, we define several tools needed to prove that graphs $G \in \mathcal{C}_t$ have bounded treewidth. In Section \ref{sec:bs-in-cbags}, we construct a useful induced subgraph $\beta$ of $G$ and prove that $\beta$ has bounded treewidth. Finally, in Section \ref{sec:extending-bs}, we use that $\beta$ has bounded treewidth to prove that $G$ has bounded treewidth.

\section{Tools}
\label{sec:tools}
In this section, we describe the tools needed to prove that graphs in $\mathcal{C}_t$ have bounded treewidth. These tools fall into four categories: balanced separators, separations, central bags, and cutsets obtained from wheels. 
\subsection{Balanced separators} 
Let $G$ be a graph. A {\em weight function on $G$} is a function $w:V(G) \to \mathbb{R}$. For $X \subseteq V(G)$, we let $w(X) = \sum_{x \in X} w(x)$. 
Let $G$ be a graph, let $w: V(G) \to [0, 1]$ be a weight function on $G$ with $w(G) = 1$, and let $c \in [\frac{1}{2}, 1)$. A set $X \subseteq V(G)$ is a  {\em $(w, c)$-balanced separator} if $w(D) \leq c$ for every component $D$ of $G \setminus X$. The next two lemmas state how $(w, c)$-balanced separators relate to treewidth. The first result was originally proven by Harvey and Wood in \cite{params-tied-to-tw} using different language, and was restated and proved in the language of $(w, c)$-balanced separators in \cite{wallpaper}. 
\begin{lemma}[\cite{wallpaper}, \cite{params-tied-to-tw}]\label{lemma:bs-to-tw}
Let $G$ be a graph, let $c \in [\frac{1}{2}, 1)$, and let $k$ be a positive integer. If $G$ has a $(w, c)$-balanced separator of size at most $k$ for every weight function $w: V(G) \to [0, 1]$ with $w(G) = 1$, then $\tw(G) \leq \frac{1}{1-c}k$. 
\end{lemma}

\begin{lemma}[\cite{cygan}]
\label{lemma:tw-to-weighted-separator}
Let $G$ be a graph and let $k$ be a positive integer. If $\tw(G) \leq k$, then $G$ has a $(w, c)$-balanced separator of size at most $k+1$ for every $c \in [\frac{1}{2}, 1)$ and for every weight function $w:V(G) \to [0, 1]$ with $w(G) = 1$.
\end{lemma}
\subsection{Separations} 
A {\em separation} of a graph $G$ is a triple $(A, C, B)$, where $A, B, C \subseteq V(G)$, $A \cup C \cup B = V(G)$, $A$, $B$, and $C$ are pairwise disjoint, and $A$ is anticomplete to $B$. If $S = (A, C, B)$ is a separation, we let $A(S) = A$, $B(S) = B$, and $C(S) = C$. Two separations $(A_1, C_1, B_1)$ and $(A_2, C_2, B_2)$ are {\em nearly non-crossing} if every component of $A_1 \cup A_2$ is a component of $A_1$ or a component of $A_2$. A separation $(A, C, B)$ is a {\em star separation} if there exists $v \in C$ such that $C \subseteq N[v]$. Let $S_1 = (A_1, C_1, B_1)$ and $S_2 = (A_2, C_2, B_2)$ be separations of $G$. We say $S_1$ is a {\em shield for $S_2$} if $B_1 \cup C_1 \subseteq B_2 \cup C_2$. 

\begin{lemma}
\label{lemma:shields}
Let G be a ($C_4$, diamond)-free graph with no clique cutset, let $v_1, v_2 \in V(G)$, and let $S_1 = (A_1, C_1, B_1)$ and $S_2 = (A_2, C_2, B_2)$ be star separations of $G$ such that $v_i \subseteq C_i \subseteq N[v_i]$, $B_i$ is connected, and $N(B_i) = C_i \setminus \{v_i\}$ for $i = 1, 2$. Suppose that $v_2 \in A_1$ and $B_2 \cap (B_1 \cup (C_1 \setminus \{v_1\})) \neq \emptyset$. Then, $S_1$ is a shield for $S_2$. 
\end{lemma}
\begin{proof}
Since $v_2 \in A_1$ and $A_1$ is anticomplete to $B_1$, it follows that $C_2 \subseteq A_1 \cup C_1$, and thus $B_1$ is contained in a connected component of $G \setminus C_2$. First, we show that $B_1 \subseteq B_2$. If there exists $x \in B_1 \cap B_2$, then it holds that $B_1 \subseteq B_2$, so we may assume that $B_1 \cap B_2 = \emptyset$. Consequently, $B_1 \subseteq A_2$ and there exists $x \in (C_1 \setminus \{v_1\}) \cap B_2$. But $x \in B_2$, $B_1 \subseteq A_2$, and $A_2$ is anticomplete to $B_2$, a contradiction. This proves that $B_1 \subseteq B_2$. 

Since every vertex of $C_1 \setminus \{v_1\}$ has a neighbor in $B_1$, and thus in $B_2$, it follows that $C_1 \setminus \{v_1\} \subseteq B_2 \cup C_2$. Now consider $v_1$. If there exists $x \in C_1 \setminus \{v_1\}$ such that $x \in B_2$, then $v_1$ has a neighbor in $B_2$ and so $v_1 \in B_2 \cup C_2$, as required. Thus we may assume that $C_1 \setminus \{v_1\} \subseteq C_2$, and so $v_2$ is complete to $C_1 \setminus \{v_1\}$. 

Since $G$ has no clique cutset and $N(B_1) = C_1$, there exist $x, y \in C_1 \setminus \{v_1\}$ such that $x$ and $y$ are non-adjacent. But now $\{x, y, v_1, v_2\}$ is either a diamond or a $C_4$, a contradiction. 
\end{proof}

Let $G$ be a graph and let $w: V(G) \to [0, 1]$ be a weight function on $G$ with $w(G) = 1$. A vertex $v \in V(G)$ is called {\em balanced} if $w(D) \leq \frac{1}{2}$ for every component $D$ of $G \setminus N[v]$, and {\em unbalanced} otherwise. Let $U$ denote the set of unbalanced vertices of $G$. Let $v \in U$. The {\em canonical star separation for $v$}, denoted $S_v = (A_v, C_v, B_v)$, is defined as follows: $B_v$ is the connected component of $G \setminus N[v]$ with largest weight, $C_v = \{v\} \cup N(v) \cap N(B_v)$, and $A_v = V(G) \setminus (B_v \cup C_v)$. Note that $B_v$ is well-defined since $v \in U$. 

Let $\leq_A$ be the relation on $U$ where for $x, y \in U$, $x \leq_A y$ if and only if $x = y$ or $y \in A_x$. 

\begin{lemma}
\label{lemma:leqA-partial-order}
Let $G$ be a ($C_4$, diamond)-free graph with no clique cutset, let $w:V(G) \to [0, 1]$ be a weight function on $G$ with $w(G) = 1$, let $U$ be the set of unbalanced vertices of $G$, and let $\leq_A$ be the relation on $U$ defined above. Then, $\leq_A$ is a partial order. 
\end{lemma}
\begin{proof}
We will show that $\leq_A$ is reflexive, antisymmetric, and transitive. The relation is reflexive by definition. Let $x, y \in U$ be such that $x \neq y$ and suppose that $x \leq_A y$. By Lemma \ref{lemma:shields}, it holds that $S_x$ is a shield for $S_y$, and so $B_x \cup C_x \subseteq B_y \cup C_y$. But $x \in C_x$, so $x \not \in A_y$. Since $x \neq y$, it follows that $y \not \leq_A x$, and so the relation is antisymmetric.

Finally, suppose that $x, y, z \in U$ such that $x \leq_A y$ and $y \leq_A z$, so $y \in A_x$ and $z \in A_y$. By Lemma \ref{lemma:shields}, it follows that $S_x$ is a shield for $S_y$, so $B_x \cup C_x \subseteq B_y \cup C_y$. Since $z \in A_y$, it follows that $z \not \in B_x \cup C_x$, so $z \in A_x$. Therefore, $x \leq_A z$, and the relation is transitive.
\end{proof} 
\subsection{Central bags} 
Let $G$ be a graph. We call a collection $\S$ of separations of $G$ {\em smooth} if the following hold: 
\begin{enumerate}[(i)]
    \item $S_1$ and $S_2$ are nearly non-crossing for all distinct $S_1, S_2 \in \S$; 
    \item There is a set of vertices $v(\S) \subseteq V(G)$ such that there is a bijection $f$ from $v(\S)$ to $\S$ with $v \in C(f(v)) \subseteq N[v]$;
    \item \label{smooth:v-cap-A-empty} $v(\S) \cap A(S) = \emptyset$ for all $S \in \S$. 
\end{enumerate} 

Let $\S$ be a smooth collection of separations of $G$. Then, the {\em central bag for $\S$}, denoted $\beta_\S$, is defined as follows: 
$$\beta_\S = \bigcap_{S \in \S} (B(S) \cup C(S)).$$

Let $G$ be a graph and let $w:V(G) \to [0, 1]$ be a weight function on $G$ with $w(G) = 1$. Let $\S$ be a smooth collection of separations of $G$, and let $\beta_\S$ be the central bag for $\S$. By property \eqref{smooth:v-cap-A-empty} of smooth collections of separations, it holds that $v(\S) \subseteq \beta_\S$. We now define the {\em inherited weight function $w_\S$ on $\beta_\S$} as follows. Fix an ordering $\{v_1, \hdots, v_k\}$ of $v(\S)$. For every $f(v_i) \in \S$, let $A^*(f(v_i))$ be the union of all connected components $D$ of $\bigcup_{1 \leq i \leq k} A(f(v_i))$ such that $D \not \subseteq A(f(v_j))$ for every $j < i$. In particular, $(A^*(f(v_1)), \hdots, A^*(f(v_k))$ is a partition of $\bigcup_{S \in \S} A(S)$.  Now, $w_\S(v_i) = w(v_i) + w(A^*(f(v_i)))$ for all $v_i \in v(\S)$, and $w_\S(v) = w(v)$ for all $v \not \in v(\S)$.

 \begin{lemma}
 \label{lemma:grow-a-separator}
 Let $G$ be a ($C_4$, diamond)-free graph with no clique cutset, let $w: V(G) \to [0, 1]$ be a weight function on $G$ with $w(G) = 1$, and let $c \in [\frac{1}{2}, 1)$. Let $\S$ be a smooth collection of separations of $G$, let $\beta_\S$ be the central bag for $\S$, and let $w_\S$ be the inherited weight function on $\beta_\S$. Suppose that $X \subseteq \beta_\S$ is a $(w_\S, c)$-balanced separator of $\beta_\S$. Then, $Y = X \cup (N[X \cap v(\S)] \cap \beta_\S)$ is a $(w, c)$-balanced separator of $G$.
 \end{lemma}
 \begin{proof}
Let $Q_1, \hdots, Q_m$ be the connected components of $\beta_\S \setminus X$. 
 Let $A_i = \bigcup_{v_j \in Q_i \cap v(\S)} A^*(f(v_j))$ for $1 \leq i \leq m$. Since $N(A^*(f(v)) \subseteq C_v$ and $C_v \subseteq N[v]$ for all $v \in v(\S)$, it follows that for every connected component $D'$ of $G \setminus Y$, either $D' \subseteq Q_i \cup A_i$ or $D' \subseteq A^*(f(v))$ for some $v \in v(\S) \cap X$. Let $D'$ be a component of $G \setminus Y$. Since $A^*(f(v)) \subseteq A_v$ and $w(A_v) \leq \frac{1}{2}$ for every unbalanced vertex $v \in V(G)$, if $D' \subseteq A^*(f(v))$ for some $v \in v(\S)$, then $w(D') \leq \frac{1}{2} \leq c$, so we may assume $D' \not \subseteq A^*(f(v))$. Therefore, $D' \subseteq A_i$ for some $1 \leq i \leq m$. By the definition of $w_\S$, it holds that 
 \begin{align*}
     w_\S(Q_i) &= w(Q_i) +\sum_{v \in Q_i \cap v(\S)} w(A^*f(v))\\
    &= w(Q_i) + w(A_i).
 \end{align*} Since $w_\S(Q_i) \leq c$, it follows that $w(Q_i \cup A_i) \leq c$. Since $D' \subseteq Q_i \cup A_i$, it follows that $w(D') \leq c$ for every connected component $D'$ of $G \setminus Y$. 
 \end{proof}
 
 \subsection{Wheels}
Recall that a {\em wheel} $(H, w)$ is a hole $H$ and a vertex $w \in V(G)$ such that $w$ has at least three pairwise non-adjacent neighbors in $H$. If $(H, w)$ is a wheel, a {\em sector of $(H, w)$} is a path $P \subseteq H$ of length at least one such that the ends of $P$ are adjacent to $w$ and $P^*$ is anticomplete to $w$. A sector of $(H, w)$ is {\em long} if it has length greater than one. 

\begin{lemma}
\label{lemma:common_nbrs}
Let $G$ be an even-wheel-free graph, let $H$ be a hole of $G$, and let $v_1, v_2 \in V(G)$ be adjacent vertices each with at least two non-adjacent neighbors in $H$. Then, $v_1$ and $v_2$ have a common neighbor in $H$. 
\end{lemma}
\begin{proof}
Suppose that $v_1$ and $v_2$ have no common neighbors in $H$. Let $Q \subseteq H$ be a long sector of $(H, v_1)$. Then, $Q \cup \{v_1\}$ is a hole. Since $v_2$ is adjacent to $v_1$ and $G$ is even-wheel-free, it follows that $v_2$ has an odd number of neighbors in $Q \cup \{v_1\}$, and thus $v_2$ has an even number of neighbors in $Q$. Since $v_1$ and $v_2$ have no common neighbors in $H$, every neighbor of $v_2$ in $H$ is in the interior of a sector of $(H, v_1)$. Therefore, $v_2$ has an even number of neighbors in $H$. Since $G$ is even-wheel-free, $v$ has exactly two neighbors in $H$ and they are adjacent, a contradiction. 
\end{proof}


 A wheel $(H, w)$ is a {\em twin wheel} if $N(w) \cap H$ is a path of length two. A wheel $(H, w)$ is a {\em short pyramid} if $|N(w) \cap H| = 3$ and $w$ has exactly two adjacent neighbors in $H$. A {\em proper wheel} is a wheel that is not a twin wheel or a short pyramid.  A wheel $(H, w)$ is a {\em universal wheel} if $w$ is complete to $H$. We will use the following result about wheels and star cutsets. In \cite{daSilvaVuskovic}, Lemma \ref{lemma:forcer_lemma} is proven for a class of graphs called {\em $C_4$-free odd-signable graphs}. It is also shown in \cite{daSilvaVuskovic} that ($C_4$, even wheel, theta, prism)-free graphs are $C_4$-free odd-signable graphs. Since we are interested in ($C_4$, even wheel, theta, prism)-free graphs, we state Lemma \ref{lemma:forcer_lemma} about ($C_4$, even wheel, theta, prism)-free graphs. 

\begin{lemma}[\cite{daSilvaVuskovic}]
Let $G$ be a ($C_4$, even wheel, theta, prism)-free graph that contains a proper wheel $(H, x)$
that is not a universal wheel. Let $x_1$ and $x_2$ be the endpoints of a long sector $Q$ of $(H, x)$. Let $W$
be the set of all vertices $h \in H \cap N(x)$ such that the subpath of $H \setminus \{x_1\}$ from $x_2$ to $h$ contains
an even number of neighbors of $x$, and let $Z = H \setminus (Q \cup N(x))$. Let $N' = N(x) \setminus W$. Then,
$N' \cup \{x\}$ is a cutset of $G$ that separates $Q^*$
from $W \cup Z$.
\label{lemma:forcer_lemma}
\end{lemma}
In particular, Lemma \ref{lemma:forcer_lemma} implies the following. 
\begin{lemma}
\label{lemma:no_wheels}
Let $G$ be a ($C_4$, even wheel, theta, prism)-free graph and let $(H, v)$ be a proper wheel of $G$ that is not a universal wheel. Let $(A, C, B)$ be a separation of $G$ such that $v \in C \subseteq N[v]$, $B$ is connected, and $N(B) = C \setminus \{v\}$. Then, $H \not \subseteq B \cup C$. 
\end{lemma}
\begin{proof}
By Lemma \ref{lemma:forcer_lemma}, there exist $x, y \in H$ such that there is no path from $x$ to $y$ with interior anticomplete to $v$. Suppose that $x, y \in B \cup C$. Since $B$ is connected and every vertex of $C \setminus \{v\}$ has a neighbor in $B$, it follows that there is a path $P$ from $x$ to $y$ with $P^* \subseteq B$. But $N(B) = C \setminus \{v\}$, and so $v$ is anticomplete to $P^*$, a contradiction. It follows that $\{x, y\} \not \subseteq B \cup C$, and so $H \not \subseteq B \cup C$. 
\end{proof}

 \section{Balanced separators and central bags}
 \label{sec:bs-in-cbags}
In this section, we construct a useful central bag for graphs in $\mathcal{C}_t$ and prove that the central bag has bounded treewidth. First, we state the observation that clique cutsets do not affect treewidth (this is a special case of Lemma 3.1 from \cite{clique-cutsets-tw}). 
 
\begin{lemma}
\label{lemma:clique-cutsets-tw}
Let $G$ be a graph. Then, the treewidth of $G$ is equal to the maximum treewidth over all induced subgraphs of $G$ with no clique cutset. 
\end{lemma}

Because of Lemma \ref{lemma:clique-cutsets-tw}, we often assume that the graphs we work with do not have clique cutsets. The next lemma examines how three vertices can have neighbors in a connected subgraph. 
 \begin{lemma}[\cite{wallpaper}]
\label{lemma:three_vtx_attachments} 
Let $x_1, x_2, x_3$ be three distinct vertices of a graph $G$. Assume that $H$ is a connected
induced subgraph of $G \setminus \{x_1, x_2, x_3\}$ such that $H$ contains at least one neighbor of each of $x_1, x_2,
x_3$, and that subject to these conditions $V(H)$ is minimal subject to inclusion. Then one of the
following holds:
\begin{enumerate}[(i)]
    \item For distinct $i, j, k \in \{1, 2, 3\}$, there exists $P$ that is either a path from $x_i$ to $x_j$ or a
hole containing the edge $x_ix_j$ such that
\begin{itemize}
\item $H = P \setminus \{x_i, x_j\}$, and
\item either $x_k$ has at least two non-adjacent neighbors in $H$ or $x_k$ has exactly two neighbors
in H and its neighbors in H are adjacent.
\end{itemize}
\item There exists a vertex $a \in H$ and three paths $P_1, P_2, P_3$, where $P_i$ is from $a$ to $x_i$, such that
\begin{itemize}
\item $H = (P_1 \cup P_2 \cup P_3) \setminus \{x_1, x_2, x_3\}$, and
\item the sets $P_1 \setminus \{a\}$, $P_2 \setminus \{a\}$ and $P_3 \setminus \{a\}$ are pairwise disjoint, and
\item for distinct $i, j \in \{1, 2, 3\}$, there are no edges between $P_i \setminus \{a\}$ and $P_j \setminus \{a\}$, except
possibly $x_ix_j$.
\end{itemize}
\item There exists a triangle $a_1a_2a_3$ in $H$ and three paths $P_1, P_2, P_3$, where $P_i$ is from $a_i$ to $x_i$,
such that
\begin{itemize}
\item $H = (P_1 \cup P_2 \cup P_3) \setminus \{x_1, x_2, x_3\}$, and
\item the sets $P_1$, $P_2$, and $P_3$ are pairwise disjoint, and
\item for distinct $i, j \in \{1, 2, 3\}$, there are no edges between $P_i$ and $P_j$ , except $a_ia_j$ and
possibly $x_ix_j$. 
\end{itemize}
\end{enumerate}
\end{lemma}


Using Lemma \ref{lemma:three_vtx_attachments}, we prove the following theorem, which can also be easily deduced from Lemma 1.1 of \cite{isk4}. 

\begin{theorem}
\label{thm:theta-triangle-wheel-free}
Let $G$ be a (theta, triangle, wheel)-free graph. Then, $\tw(G) \leq 2$. 
\end{theorem}
\begin{proof}
By Lemma \ref{lemma:clique-cutsets-tw}, we may assume $G$ does not have a clique cutset (so in particular, $G$ is connected).

\sta{\label{no-star-cutset} $G$ does not have a star cutset $C$ with $v \in C \subseteq N[v]$ for some $v \in V(G)$.}

Suppose that $v \in V(G)$ and $G$ has a star cutset $C \subseteq N[v]$ with $v \in C$. Since $G$ does not have a clique cutset, there exist $x, y \in C \setminus \{v\}$ and two connected components $D_1, D_2$ of $G \setminus C$ such that $\{x, y\} \subseteq N(D_1) \cap N(D_2)$ (and since $G$ is triangle-free, $x$ and $y$ are non-adjacent). Let $P_1$ be a path from $x$ to $y$ with $P_1^* \subseteq D_1$, let $P_2$ be a path from $x$ to $y$ with $P_2^* \subseteq D_2$, and let $H$ be the hole given by $P_1 \cup P_2$. Since $x$ and $y$ are non-adjacent, $v$ has at least two non-adjacent neighbors in $H$. If $N(v) \cap H = \{x, y\}$, then $G$ contains a theta between $x$ and $y$ through $x \dd v \dd y$, $x \dd P_1 \dd y$, and $x \dd P_2 \dd y$, a contradiction, so $v$ has at least three neighbors in $H$. Since $G$ is triangle-free, the neighbors of $v$ are pairwise non-adjacent. But now $(H, v)$ is a wheel of $G$, a contradiction. This proves \eqref{no-star-cutset}.

\sta{\label{max-degree-2} $\deg(v) \leq 2$ for all $v \in V(G)$.} 

Let $v \in V(G)$ and let $B = G \setminus N[v]$. By \eqref{no-star-cutset}, $B$ is connected, and since $G$ does not have a clique cutset, it follows that $N(B) = N(v)$. Suppose that $x, y, z \in N(v)$, and let $H \subseteq B$ be inclusion-wise minimal such that $H$ contains a neighbor of $x, y,$ and $z$. We apply Lemma \ref{lemma:three_vtx_attachments}. If case (ii) holds, then $H \cup \{v, x, y, z\}$ is a theta, a contradiction. Case (iii) does not hold because $G$ is triangle-free. Therefore, case (i) holds. Then, up to symmetry between $x, y, z$, $H \cup \{x, z\}$ is a path from $x$ to $z$ such that $y$ has two non-adjacent neighbors in $H$. But now $H' = \{v, x, z\} \cup H$ is a hole and $y$ has three pairwise non-adjacent neighbors in $H'$, so $(H', v)$ is a wheel of $G$, a contradiction. This proves \eqref{max-degree-2}. \\ 

By \eqref{max-degree-2}, it follows that $G$ is either a path or a cycle, and thus $\tw(G) \leq 2$.  
\end{proof}

Next, we give a simple characterization of the neighborhood of vertices in diamond-free graphs. 
\begin{lemma}
\label{lemma:clique-nbrs}
Let $G$ be diamond-free and let $v \in V(G)$. Then, $N(v)$ is the union of disjoint, pairwise anticomplete cliques. 
\end{lemma}
\begin{proof}
If $N(v)$ contains $P_3$, then $v \cup N(v)$ contains a diamond, a contradiction. Therefore, $N(v)$ is $P_3$-free, and thus the union of disjoint, pairwise anticomplete cliques. 
\end{proof}

For $X \subseteq V(G)$, let $\Hub(X)$ denote the set of all vertices $x \in X$ for which there exists a wheel $(H, x)$ with $H \subseteq X$. 

\begin{lemma}
\label{lemma:bounding-nbrhood-helper}
Let $G$ be a ($C_4$, theta, prism, even wheel, diamond)-free graph with no clique cutset and let $w:V(G) \to [0, 1]$ be a weight function on $G$ with $w(G) = 1$. Let $\S$ be a smooth collection of separations of $G$, let $\beta_\S$ be the central bag for $\S$, and let $w_\S$ be the inherited weight function on $\beta_\S$. Let $v \in \beta_\S$ and (by Lemma \ref{lemma:clique-nbrs}) let $N_{\beta_\S}(v) \setminus \Hub(\beta_\S) =K_1 \cup \cdots \cup K_t$, where $K_1, \hdots, K_t$ are disjoint, pairwise anticomplete cliques. Assume that $v$ is not a pyramid apex in $\beta_\S$. Let $D$ be a component of $\beta_\S \setminus N[v]$. Then, at most two of $K_1, \hdots, K_t$ have a neighbor in $D$. 
\end{lemma}
\begin{proof}
Suppose that $\{a, b, c\} \subseteq \{1, \hdots, t\}$ such that $K_a, K_b, K_c$ each have a neighbor in $D$. Let $x_1 \in K_a$, $x_2 \in K_b$, and $x_3 \in K_c$ be such that $x_1, x_2, x_3$ have neighbors in $D$. Note that $\{x_1, x_2, x_3\}$ is independent. Let $H \subseteq D$ be inclusion-wise minimal such that $x_1, x_2, x_3$ each have a neighbor in $H$. We apply Lemma \ref{lemma:three_vtx_attachments}. If case (ii) holds, then $\{v, x_1, x_2, x_3\} \cup H$ is a theta, a contradiction. If case (iii) holds, then $\{v, x_2, x_2, x_3\} \cup H$ is a pyramid of $\beta_\S$ with apex $v$, a contradiction. Therefore, up to symmetry between $x_1, x_2, x_3$, it holds that that $H \cup \{x_1, x_3\}$ is a path from $x_1$ to $x_3$. If $x_2$ has two non-adjacent neighbors in $H$, then $x_2$ is a wheel center for the hole given by $H \cup \{v\}$, a contradiction (since $x_2 \in N_{\beta_\S}(v) \setminus \Hub(\beta_\S)$). Therefore, $x_2$ has exactly two adjacent neighbors in $H$. But now $\{v, x_2\} \cup H$ is a pyramid of $\beta_\S$ with apex $v$, a contradiction.
\end{proof}

We next prove a result about balanced separators in central bags with balanced vertices. By $\omega(G)$ we denote the size of the largest clique of $G$.
\begin{lemma}
\label{lemma:balanced_vtx_bs}
Let $G$ be a ($C_4$, theta, pyramid, prism, diamond, even wheel)-free graph with no clique cutset, let $w:V(G) \to [0, 1]$ be a weight function on $G$ with $w(G) = 1$, and let $c \in [\frac{1}{2}, 1)$.
Let $\S$ be a smooth collection of separations of $G$, let $\beta_\S$ be the central bag for $\S$, and let $w_\S$ be the inherited weight function on $\beta_\S$. Suppose that there exists $v \in \beta_\S \setminus v(\S)$ such that $v$ is balanced in $G$, and assume that $|N(v) \cap \Hub(\beta_\S)| <C$. Then $\beta_\S$ has a $(w_\S, c)$-balanced
separator of size at most $6\omega(\beta_\S) + C$.
\end{lemma}
\begin{proof} Let $D_1, \hdots, D_m$ be the components of $\beta_\S \setminus N[v]$. 

\sta{\label{D-small-wt} $w_\S(D_i) \leq \frac{1}{2}$ for all $1 \leq i \leq m$.}

Let $D'$ be a component of $G \setminus N[v]$, and let $s \in v(\S) \cap D'$. Suppose that $u \in A(f(s)) \cap N(v)$. Then, since $v \in N(A(f(s)))$, it follows that $v \in A(f(s)) \cup C(s)$. But $v \in \beta_{\S}$, so $v \not \in A(f(s))$ because $\S$ is smooth, and $v$ is not adjacent to $s$, a contradiction. This proves that $A(f(s)) \cap N(v) = \emptyset$. Now, let $D'' \neq D'$ be a component of $G \setminus N[v]$, and suppose that $D'' \cap A(f(s)) \neq \emptyset$. Then, since $N(s) \cap D'' = \emptyset$, it follows that $D'' \subseteq A(f(s))$ and $N(D'') \subseteq N(v) \cap C(f(s))$.  Since $G$ has no clique cutset, it follows that there exist $x, y \in N(v) \cap C(f(s))$ with $x$ and $y$ non-adjacent. But now $\{s, v, x, y\}$ is a $C_4$, a contradiction. This proves that $D'' \subseteq B(f(s))$.

Let $D$ be a component of $\beta_\S \setminus N[v]$ and let $D'$ be the component of $G \setminus N[v]$ containing $D$. Now, for every $s \in v(\S) \cap D$, it holds that $A(f(s)) \subseteq D'$. 
It follows that $w_\S(D) \leq w(D')$. Since $v$ is balanced in $G$, it holds that  $w(D') \leq \frac{1}{2}$, and thus \eqref{D-small-wt} follows. \\

By Lemma \ref{lemma:clique-nbrs}, let $N_\beta(v) \setminus \Hub(\beta_\S) = K_1 \cup \cdots \cup K_t$, where $K_i$ is a clique for $1 \leq i \leq t$ and $K_1, \hdots, K_t$ are disjoint and pairwise anticomplete to each other. Let $H$ be a graph with vertex set $\{k_1, \hdots, k_t, d_1, \hdots, d_m\}$, where $H$ contains an edge between $k_i$ and $d_j$ if and only if $K_i$ has a neighbor in $D_j$.



\sta{\label{H-theta-wheel-bipartite} $H$ is (theta, wheel)-free and bipartite.} 

By definition of $H$, $(\{k_1, \hdots, k_t\}, \{d_1, \hdots, d_m\})$ is a bipartition of $H$, so $H$ is bipartite. By Lemma \ref{lemma:bounding-nbrhood-helper}, $\deg(d_i) \leq 2$ for $1 \leq i \leq m$, so $H$ is wheel-free. Suppose that $H$ contains a theta between $x$ and $y$ (with $x, y \in V(H)$). Since $x$ and $y$ have degree at least three, it follows that $x, y \in \{k_1, \hdots, k_t\}$. Let $P_1, P_2, P_3$ be three disjoint, pairwise anticomplete paths from $x$ to $y$ in $H$. Let $Q_1, Q_2, Q_3$ be paths in $\beta_\S$, where $Q_i$ has ends $x_i$ and $y_i$ and is formed from $P_i$ by replacing every path $k_i \dd d_j \dd k_\ell \subseteq P_i$ with a path from $K_i$ to $K_\ell$ through $D_j$ in $\beta_\S$. Note that since $x, y \in \{k_1, \hdots, k_t\}$, it follows that $\{x_1, x_2, x_3\} \subseteq K_i$ and $\{y_1, y_2, y_3\} \subseteq K_j$ for some $1 \leq i, j \leq t$.  

Let $x_i'$ be the neighbor of $x_i$ in $Q_i$ for $i = 1, 2, 3$. Let $H_1$ be an inclusion-wise minimal connected subset of $K_i$ such that $H_1$ contains neighbors of $x_1', x_2', x_3'$. We apply Lemma \ref{lemma:three_vtx_attachments} to $x_1', x_2', x_3'$ and $H_1$. (Note that while $x_1, x_2, x_3$ are not necessarily distinct, $x_1', x_2', x_3'$ are distinct). Similarly, let $y_i'$ be the neighbor of $y_i$ in $Q_i$ for $i = 1, 2, 3$. Let $H_2$ be an inclusion-wise minimal connected subset of $K_j$ such that $H_2$ contains neighbors of $y_1', y_2', y_3'$. We apply Lemma \ref{lemma:three_vtx_attachments} to $y_1', y_2', y_3'$, and $H_2$. Note that $\{x_1', x_2', x_3'\} \subseteq \bigcup_{1 \leq i \leq m} D_i$ and that $x_1', x_2', x_3'$ are each in distinct components $D_i$. Similarly, $\{y_1', y_2', y_3'\} \subseteq \bigcup_{1 \leq i \leq m} D_i$ and $y_1', y_2', y_3'$ are each in distinct components $D_i$. Therefore, $\{x_1', x_2', x_3'\}$ are pairwise non-adjacent and $\{y_1', y_2', y_3'\}$ are pairwise non-adjacent.  

Suppose that $\{x_1', x_2', x_3'\} \cup H_1$ satisfies condition (i) of Lemma \ref{lemma:three_vtx_attachments}. If $\{y_1', y_2', y_3'\} \cup H_2$ satisfies condition (i), then $H_1 \cup H_2 \cup \{x_1', x_2', x_3', y_1', y_2', y_3'\}$ is a prism or a line wheel, a contradiction. If $\{y_1', y_2', y_3'\} \cup H_2$ satisfies condition (ii), then $H_1 \cup H_2 \cup \{x_1', x_2', x_3', y_1', y_2', y_3'\}$ is a pyramid, a contradiction. If $\{y_1', y_2', y_3'\} \cup H_2$ satisfies condition (iii), then $H_1 \cup H_2 \cup \{x_1', x_2', x_3', y_1', y_2', y_3'\}$ is a prism or a line wheel, a contradiction. Therefore, up to symmetry, neither $\{x_1', x_2', x_3'\} \cup H_1$ nor $\{y_1', y_2', y_3'\} \cup H_2$ satisfy condition (i). 

Suppose that $\{x_1', x_2', x_3'\} \cup H_1$ satisfies condition (ii) of Lemma \ref{lemma:three_vtx_attachments}. If $\{y_1', y_2', y_3'\} \cup H_2$ satisfies condition (ii), then $H_1 \cup H_2 \cup \{x_1', x_2', x_3', y_1', y_2', y_3'\}$ is a theta in $\beta$, a contradiction. If $\{y_1', y_2', y_3'\} \cup H_2$ satisfies condition (iii), then $H_1 \cup H_2 \cup \{x_1', x_2', x_3', y_1', y_2', y_3'\}$ is a pyramid, a contradiction. Therefore, up to symmetry, neither $\{x_1', x_2', x_3'\} \cup H_1$ nor $\{y_1', y_2', y_3'\} \cup H_2$ satisfy condition (ii). 

It follows that $\{x_1', x_2', x_3'\} \cup H_1$ and $\{y_1', y_2', y_3'\} \cup H_2$ satisfy condition (iii). But now $H_1 \cup H_2 \cup \{x_1', x_2', x_3', y_1', y_2', y_3'\}$ is a prism or a line wheel, a contradiction. This proves \eqref{H-theta-wheel-bipartite}. \\

Let $w_H: V(H) \to [0, 1]$ be a weight function on $H$ defined as follows: $w_H(k_i) = w_\S(K_i)$ for $1 \leq i \leq t$, and $w_H(d_i) = w_\S(D_i)$ for $1 \leq i \leq m$. Let $\overline{w_H}$ be a weight function on $H$ such that $\overline{w_H}(v) = \frac{w_H(v)}{w_H(H)}$ for every $v \in H$. Note that $w_H(H) \leq 1$, so for all $X \subseteq V(H)$, it holds that $\overline{w_H}(X) \geq w_H(X)$. 

\sta{\label{betaS-has-bs} $\beta_\S$ has a $(w_\S, c)$-balanced separator of size at most $6\omega(\beta_\S) + C$.}

Since $H$ is (theta, triangle, wheel)-free, it follows from Theorem \ref{thm:theta-triangle-wheel-free} and Lemma \ref{lemma:tw-to-weighted-separator} that $H$ has a $(\overline{w_H}, \frac{1}{2})$-balanced separator $X$ of size at most $3$. Let $Y = \{v\} \cup (N(v) \cap \Hub(\beta_\S)) \cup \{K_i \ \text{s.t.} \ X \cap N[k_i] \neq \emptyset\}.$ By assumption of the lemma, $N(v) \cap \Hub(\beta_\S) < C$, and by Lemma \ref{lemma:bounding-nbrhood-helper}, it holds that $| X \cup \{K_i \ \text{s.t.} \ X \cap N[k_i] \neq \emptyset\}| \leq 2\omega(\beta_S)|X| \leq  6\omega(\beta_\S)$. Therefore, $|Y| \leq 6\omega(\beta_\S) + C$. 

We claim that $Y$ is a $(w_\S, c)$-balanced separator of $\beta_\S$. Let $F$ be a component of $\beta_\S \setminus Y$. Note that by construction of $Y$, it holds that $F \subseteq \bigcup_{1 \leq i \leq t} K_t \cup \bigcup_{1 \leq i \leq m} D_i$, that if $K_i \cap F \neq \emptyset$, then $K_i \subseteq F$, and if $D_i \cap F \neq \emptyset$, then $D_i \subseteq F$. Let $K_F = \{i \text{ s.t. } K_i \subseteq F\}$ and $D_F = \{i \text{ s.t. } D_i \subseteq F\}$. If $F = D_i$ for some $1 \leq i \leq m$, then it follows from \eqref{D-small-wt} that $w_\S(F) \leq \frac{1}{2}$, so we may assume that $K_F \neq \emptyset$. Let $i \in D_F$ and suppose that $d_i \in X$. Then, $N(D_i) \subseteq Y$, so $K_F = \emptyset$, a contradiction. Therefore, $d_i \not \in X$. Similarly, for $i \in K_F$, it holds that $k_i \not \in X$, since $K_i \not \subseteq Y$. Let $Q = \{k_i \text{ s.t. } i \in K_F\} \cup \{d_i \text{ s.t. } i \in D_F\}$. Then, $Q$ is contained in a connected component of $H \setminus X$, so $\overline{w_H}(Q) \leq \frac{1}{2}$. Finally, 
$$w_\S(F) = \sum_{i \in K_F} w_\S(K_i) + \sum_{i \in D_F} w_\S(D_i) = w_H(Q) \leq \overline{w_H}(Q) \leq \frac{1}{2}.$$
This proves \eqref{betaS-has-bs}. \\

Now, Lemma \ref{lemma:balanced_vtx_bs} follows from \eqref{betaS-has-bs}. 
\end{proof}

Recall that $\mathcal{C}_t^*$ is the class of ($C_4$, diamond, theta, prism, even wheel, $K_t$)-free graphs with no clique cutset. Let $G \in \mathcal{C}_t^*$, let $w:V(G) \to [0, 1]$ be a weight function on $G$ with $w(G) = 1$, and let $U$ be the set of unbalanced vertices of $G$. Let $X \subseteq U$. The {\em $X$-revised collection of separations}, denoted $\tilde{\S}_X$, is defined as follows. Let $u \in X$, and let $\tilde{S}_u = (\tilde{A}_u, \tilde{C}_u, \tilde{B}_u)$ be such that $\tilde{B}_u$ is the largest weight connected component of $G \setminus N[u]$, $\tilde{C}_u = (N(u) \cap N(\tilde{B}_u)) \cup \bigcup_{v \in N(u) \cap X} (N(u) \cap N(v))$, and $\tilde{A}_u = V(G) \setminus (\tilde{C}_u \cup \tilde{B}_u)$. Then, $\tilde{\S}_X = \{\tilde{S}_u : u \in X\}$. Note that the separations in $\tilde{\S}_X$ are closely related to canonical star separations. Specifically, for all $u \in X$, the following hold: 
\begin{enumerate}[(i)]
    \item $\tilde{B}_u = B_u$, 
    \item $C_u \subseteq \tilde{C}_u \subseteq N[u]$, 
    \item $\tilde{A}_u \subseteq A_u$, 
    \item $A_u \setminus N(u) \subseteq \tilde{A}_u$.
\end{enumerate}

 \begin{lemma}
 Let $G \in \mathcal{C}_t^*$, let $w: V(G) \to [0, 1]$ be a weight function on $G$ with $w(G) = 1$, let $U$ be the set of unbalanced vertices of $G$ such that every vertex of $U$ is minimal under the relation $\leq_A$. Let $\tilde{\S} = \tilde{\S}_U$ be the $U$-revised collection of separations. Then, $\tilde{S}_u$ and $\tilde{S}_v$ are nearly non-crossing for all $\tilde{S}_u, \tilde{S}_v \in \tilde{\S}$. 
 \label{lemma:new-sepns-noncrossing}
 \end{lemma}
 \begin{proof}
 Since $\tilde{A}_x \subseteq A_x$ for all $x \in U$, it follows that either $u$ and $v$ are adjacent, $u \in \tilde{C}_v$, and $v \in \tilde{C}_u$, or $u$ and $v$ are non-adjacent, $u \in \tilde{B}_v$, and $v \in \tilde{B}_u$. Suppose first that $u$ and $v$ are non-adjacent. Then,  $u \in \tilde{B}_v$, and $v \in \tilde{B}_u$. Since $u$ is complete to $\tilde{C}_u \setminus \{u\}$, it follows that $\tilde{C}_u \subseteq \tilde{B}_v \cup \tilde{C}_v$, so $\tilde{A}_v \cap \tilde{C}_u = \emptyset$. By symmetry, $\tilde{A}_u \cap \tilde{C}_v = \emptyset$. It follows that every component of $\tilde{A}_u \cup \tilde{A}_v$ is a component of $\tilde{A}_u$ or a component of $\tilde{A}_v$, so $\tilde{S}_u$ and $\tilde{S}_v$ are nearly non-crossing. 

 Now suppose that $u$ and $v$ are adjacent. Then, $u \in \tilde{C}_v$, and $v \in \tilde{C}_u$. We may assume that there is a component of $\tilde{A}_u \cup \tilde{A}_v$ that is not a component of $\tilde{A}_u$ or a component of $\tilde{A}_v$. Therefore, there exists $u' \in \tilde{C}_u \cap \tilde{A}_v$, $v' \in \tilde{C}_v \cap \tilde{A}_u$, and a path $P$ from $u'$ to $v'$ with $P^* \subseteq \tilde{A}_u \cap \tilde{A}_v$. We claim that $u' \in C_u$ and $v' \in C_v$. Suppose that $u' \in \tilde{C}_u \setminus C_u$. Then, there exists $x \in X \cap N(u)$ such that $u'$ is complete to $\{u, x\}$. If $v$ is adjacent to $u'$, then $u' \in \tilde{C}_v$ (since $u'$ is a common neighbor of $u$ and $v$), a contradiction, so $v$ is not adjacent to $u'$ and consequently $v \neq x$. Then, since $x \in X$ and $x$ has a neighbor in $\tilde{A}_v \subseteq A_v$, it follows that $x \in C_v$, so $x$ is adjacent to $v$. But now $uu'xv$ is a diamond, a contradiction. This proves that $u' \in C_u$ and similarly $v' \in C_v$.

 Since $u, v \in U$ and $\tilde{B}_u = B_u$ and $\tilde{B}_v = B_v$, it follows that $w(\tilde{B}_u) \geq \frac{1}{2}$ and $w(\tilde{B}_v) \geq \frac{1}{2}$, and so $\tilde{B}_u \cap \tilde{B}_v \neq \emptyset$. Let $b \in \tilde{B}_u \cap \tilde{B}_v$. Since $u' \in C_u$, there exists a path $Q_1$ from $u'$ to $b$ with $Q_1 \setminus \{u'\} \subseteq \tilde{B}_u$. Similarly, there exists a path $Q_2$ from $v'$ to $b$ with $Q_2 \setminus \{v'\} \subseteq \tilde{B}_v$.  
 Now, there exists a path $Q$ from $u'$ to $v'$ with $Q \subseteq Q_1 \cup Q_2$, so in particular, $Q^* \subseteq \tilde{B}_u \cup \tilde{B}_v$. Note that $Q^* \cap (C_v \cap B_u) \neq \emptyset$ and $Q^* \cap (C_u \cap B_v) \neq \emptyset$, so in particular, $Q^*$ contains a neighbor of $u$ and a neighbor of $v$.  

 Let $H$ be the hole given by $P \cup Q$. Now, $u$ and $v$ are adjacent and each have two non-adjacent neighbors in $H$, so by Lemma \ref{lemma:common_nbrs}, $u$ and $v$ have a common neighbor in $H$. However, since $u \in \tilde{C}_v$ and $v \in \tilde{C}_u$, it follows that $N(u) \cap N(v) \subseteq \tilde{C}_u \cap \tilde{C}_v$. Since $H \cap (\tilde{C}_u \cap \tilde{C}_v) = \emptyset$, we get a contradiction. 
 \end{proof}

Next, we show how to construct a useful collection of separations of $G$. We need the following lemma: 

\begin{lemma}[\cite{logpaper}]
\label{lemma:degeneracy}
Let $G$ be (theta, $K_t$)-free with $|V(G)| = n$. Then, there exists a constant $\delta_t$ and a partition $T_1, \hdots, T_m$ of $V(G)$ into independent sets, such that $m \leq \log(n)$ and for every $i \in \{1,\hdots, m\}$ and $v \in T_i$, it holds that $|N(v) \setminus (\bigcup_{j < i} T_j)| \leq 4\delta_t$. 
\end{lemma}

We call $\delta_t$ the {\em hub constant for $t$}. Let $G \in \mathcal{C}_t$. We apply Lemma \ref{lemma:degeneracy} to $\Hub(G)$. Let $T_1, \hdots, T_\ell$ be the partition of $\Hub(G)$ given by Lemma \ref{lemma:degeneracy}. For $v \in \Hub(G)$, let $t(v)$ denote the index of the part of the partition $T_1, \hdots, T_\ell$ containing $v$. Let $v_1, \hdots, v_k$ be an ordering of $\Hub(G)$ such that for all $1 \leq i < j \leq k$, it holds that $t(v_i) \leq t(v_j)$. Let $U$ be the set of unbalanced vertices of $G$. Let $m$ be defined as follows. If $\Hub(G) \subseteq U$, then $m=k+1$. Otherwise, let $m$ ne such that $v_m$ is the minimum element of $\Hub(G) \setminus U$. Now, $\{v_1, \hdots, v_{m-1}\} \subseteq U$. Let $M$ be the set of minimal vertices of $\{v_1, \hdots, v_{m-1}\}$ under the relation $\leq_A$, and let $\tilde{S}_M$ be the $M$-revised collection of separations. We call $(\{v_1, \hdots, v_k\}, m, M, \tilde{S}_M)$ the {\em hub division} of $G$. The next two lemmas describe properties of the hub division. 

\begin{lemma}
\label{lemma:S_M-smooth}
Let $G \in \mathcal{C}_t^*$, let $w:V(G) \to [0, 1]$ be a weight function on $G$ with $w(G) = 1$, and let $(\{v_1, \hdots, v_k\}, m, M, \tilde{S}_M)$ be the hub division of $G$. Then, $\tilde{S}_M$ is a smooth collection of separations of $G$. 
\end{lemma}
\begin{proof}
By Lemma \ref{lemma:new-sepns-noncrossing}, it follows that $S_1$ and $S_2$ are nearly non-crossing for every distinct $S_1, S_2 \in \S$. By construction of $\tilde{S}_M$, there exists a set of vertices $v(\tilde{S}_M) = M$ such that there is a bijection $f$ from $v(\tilde{S}_M)$ to $\tilde{S}_M$ with $v \in C(f(v)) \subseteq N[v]$, given by $f(x) = \tilde{S}_x$ for every $x \in M$. Finally, since $M$ is minimal under the relation $\leq_A$ and $\tilde{A}_x \subseteq A_x$ for every $x \in M$, it holds that $M \cap A(\tilde{S}_x) = \emptyset$ for all $x \in M$. Therefore, $\tilde{S}_M$ is a smooth collection of separations of $G$. 
\end{proof}

By Lemma \ref{lemma:S_M-smooth}, there is a central bag $\beta_M$ for $\tilde{S}_M$ and an inherited weight function $w_M$ on $\beta_M$. 

\begin{lemma}
Let $G \in \mathcal{C}_t^*$, let $w:V(G) \to [0, 1]$ be a weight function on $G$ with $w(G) = 1$, and let $(\{v_1, \hdots, v_k\}, m, M, \tilde{S}_M)$ be the hub division of $G$. Let $\beta_M$ be the central bag for $\tilde{S}_M$ and let $w_M$ be the inherited weight function on $\beta_M$. Then, for all $1 \leq i \leq m - 1$, $v_i$ is not a wheel center of $\beta_M$. 
\label{lemma:no-wheels-in-beta}
\end{lemma}
\begin{proof}
Let $1 \leq i \leq m-1$ and suppose that $(H, v_i)$ is a wheel of $\beta_M$. 

\sta{\label{helper} $H \cap A_{v_j} \neq \emptyset$ for some $v_j \in M$.}

By Lemma \ref{lemma:no_wheels}, $H \not \subseteq B_{v_i} \cup C_{v_i}$, so $H \cap A_{v_i} \neq \emptyset$. If $v_i \in M$, then $j = i$ satisfies the statement, so we may assume $v_i \not \in M$. Then, there exists $v_j \in M$ such that $v_j \leq_A v_i$, so $v_i \in A_{v_j}$. Now, by Lemma \ref{lemma:shields}, $A_{v_i} \subseteq A_{v_j}$, and so $H \cap A_{v_j} \neq \emptyset$. This proves \eqref{helper}.  \\

By \eqref{helper}, there exists $v_j \in M$ such that $A_{v_j} \cap H \neq \emptyset$. Let $x \in A_{v_j} \cap H$. Since $v_j \in M$, it follows that $\beta_M \subseteq \tilde{B}_{v_j} \cup \tilde{C}_{v_j}$, and so $x \in \tilde{C}_{v_j}$. Let $x'$ and $x''$ be the neighbors of $x$ in $H$. Since $A_{v_j}$ is anticomplete to $B_{v_j}$ and thus $\tilde{B}_{v_j}$, it holds that $x', x'' \in \tilde{C}_{v_j}$. But $v_j$ is complete to $\tilde{C}_{v_j}$, and so $\{v_j, x, x', x''\}$ is a diamond, a contradiction. 
\end{proof}

Let $R(t, s)$ denote the minimum integer such that every graph on at least $R(t, s)$ vertices contains either
a clique of size $t$ or a stable set of size $s$. 
\begin{theorem}
\label{thm:wheel-free}
Let $G$ be a (theta, pyramid, prism, wheel, $K_t$)-free graph and let  $w:V(G) \to [0, 1]$ be a weight function on $G$. Then, $G$ has a $(w, \frac{1}{2})$-balanced separator of size at most $R(t, 4) + 1$. 
\end{theorem}
\begin{proof}
By Theorems 4.4 and 2.3 of \cite{logpaper}, it follows that $\tw(G) \leq R(t, 4) + |\Hub(G)|$.  Since $G$ is wheel-free, it follows that $|\Hub(G)| = 0$. Therefore, $\tw(G) \leq R(t, 4)$. By Lemma \ref{lemma:tw-to-weighted-separator}, $G$ has a $(w, \frac{1}{2})$-balanced separator of size at most $R(t, 4) + 1$. 
\end{proof}
Finally, we prove the main result of this section: that if $\beta_M$ is pyramid-free, then $\beta_M$ has a balanced separator of bounded size. 
\begin{theorem}
\label{thm:mainthm-betabs}
Let $G \in \mathcal{C}_t^*$, let $w:V(G) \to [0, 1]$ be a weight function on $G$, and let $(\{v_1, \hdots, v_k\},$ $m, M, \tilde{S}_M)$ be the hub division of $G$. Let $\beta_M$ be the central bag for $\tilde{S}_M$ and let $w_M$ be the inherited weight function on $\beta_M$. Assume that $\beta_M$ is pyramid-free. Then, $\beta_M$ has a $(w_M, \frac{1}{2})$-balanced separator of size at most $\max(R(t, 4) + 1, 6t + 4\delta_t)$.
\end{theorem}
\begin{proof}
First, suppose that $m = k+1$. Then, by Lemma \ref{lemma:no-wheels-in-beta}, $v$ is not a wheel center of $\beta_M$ for all $v \in \Hub(G)$. Since $\Hub(\beta_M) \subseteq \Hub(G)$, it follows that $\beta_M$ is wheel-free. By Theorem \ref{thm:wheel-free}, $\beta_M$ has a $(w_M, \frac{1}{2})$-balanced separator of size at most $R(t, 4) + 1$. 

Now, assume $m < k + 1$. We claim that $v_m \in \beta_M$. Suppose that $v_m \in A_{v_i}$ for some $v_i \in M$. Then, $N[v_m] \subseteq A_{v_i} \cup C_{v_i}$, so $B_{v_i}$ is contained in a connected component $D$ of $G \setminus N[v_m]$. Since $v_i \in M$, it follows that $v_i$ is unbalanced, so $w(B_{v_i}) \geq \frac{1}{2}$. But now $w(D) \geq \frac{1}{2}$, so $v_m$ is unbalanced, a contradiction. Therefore, $v_m \not \in A_{v_i}$ for all $v_i \in M$. Since for all $v_i \in M$ it holds that $\tilde{A}_{v_i} \subseteq A_{v_i}$, it follows that $v_m \in \tilde{B}_{v_i} \cup \tilde{C}_{v_i}$, and so $v_m \in \beta_M$.

Next, consider $N(v_m) \cap \Hub(\beta_M)$. By Lemma \ref{lemma:no-wheels-in-beta},  $\Hub(\beta_M) \subseteq \{v_{m}, v_{m+1}, \hdots, v_k\}$. Therefore, $|N(v_m) \cap \Hub(\beta_M)| \leq 4\delta_t$. Finally, by Lemma \ref{lemma:S_M-smooth}, $\tilde{S}_M$ is a smooth collection of separations of $G$. Now, by Lemma \ref{lemma:balanced_vtx_bs}, $\beta_M$ has a $(w_M, \frac{1}{2})$-balanced separator of size $6\omega(\beta_M) + 4\delta_t$.  
\end{proof}

\section{Extending balanced separators}
\label{sec:extending-bs}
In this section, we prove that we can construct a bounded balanced separator of $G$ given a bounded balanced separator of $\beta_M$. Together with the main result of the previous section, this is sufficient to prove Theorem \ref{thm:main-nonspecific}. First, we need the following lemma. 
\begin{lemma}
Let $G \in \mathcal{C}_t^*$, let $w:V(G) \to [0, 1]$ be a weight function on $G$ with $w(G) = 1$, and let $(\{v_1, \hdots, v_k\}, m, M, \tilde{S}_M)$ be the hub division of $G$. Let $\beta_M$ be the central bag for $\tilde{S}_M$. Let $v \in M$ be such that $v$ is not a pyramid apex in $\beta$. Then, $|N_{\beta_M}(v) \setminus \Hub(\beta_M)| \leq 2t$. 
\label{lemma:small-nbrs}
\end{lemma}
\begin{proof}
By Lemma \ref{lemma:clique-nbrs}, let $N_{\beta_M}(v) \setminus \Hub(\beta_M) = K_1 \cup \cdots \cup K_\ell$, where $K_1, \hdots, K_\ell$ are cliques and $K_i$ and $K_j$ are anticomplete to each other for $1 \leq i < j \leq \ell$. Recall that $B_v$ is connected and $C_v \setminus \{v\} = N(B_v)$. 

\sta{\label{path-helper} Let $k_1 \in K_1 \cap C_v \cap \beta_M$ and $k_2 \in K_2 \cap C_v \cap \beta_M$ be chosen such that the path $P$ from $k_1$ to $k_2$ with
interior in $B_v$ is as short as possible among choices of $k_1 \in K_1$, $k_2 \in K_2$. Then, there exists a path $Q$ from $k_1$ to $k_2$ with $Q^* \subseteq \beta_M \setminus N[v]$.}

If $P^* \subseteq \beta_M$, then $Q = P$ satisfies the statement, so we may assume $P^* \not \subseteq \beta_M$. Assume that $P$ is chosen such that $P^* \setminus \beta_M$ is minimal. Then, there exists $v_i \in M$ such that $P^* \cap \tilde{A}_{v_i} \neq \emptyset$. Let $a, b$ be the vertices of $\tilde{C}_{v_i} \cap P^*$ closest to $k_1$, $k_2$, respectively. Since there is a vertex of $\tilde{A}_{v_i}$ on the path from $a$ to $b$ through $P^*$, it follows that $a$ is not adjacent to $b$. If $v_i$ is not adjacent to $v$, then $P' = k_1 \dd P \dd a \dd v_i \dd b \dd P \dd k_2$ is a path from $k_1$ to $k_2$ such that either $P'$ is shorter than $P$ or $|P' \setminus \beta_M| < |P \setminus \beta_M|$, a contradiction. Therefore, $v_i$ is adjacent to $v$. Let $H$ be the hole given by $\{v\} \cup P$. Now, $v_i$ has at least three neighbors in $H$. Suppose that $v_i$ is adjacent to $k_1$. Then, $v_i \in K_1$, so $v_i$ is not adjacent to $k_2$, and $v_i \dd b \dd P \dd k_2$ is a shorter path from a vertex of $K_1$ to a vertex of $K_2$, a contradiction. It follows that $v_i$ is anticomplete to $\{k_1, k_2\}$. By Lemma \ref{lemma:forcer_lemma}, $\{k_1, k_2\} \cap A_{v_i} \neq \emptyset$, and since $v_i$ is complete to $\tilde{C}_{v_i}$ and anticomplete to $\{k_1, k_2\}$, it follows that $\{k_1, k_2\} \cap \tilde{A}_{v_i} \neq \emptyset$. But $\{k_1, k_2\} \subseteq \beta_M$, a contradiction.   This proves \eqref{path-helper}. \\

\sta{\label{comp-attaches-everywhere} There is a  component $D$ of $\beta_M \setminus N[v]$ such that for every $i \in \{1,\hdots,t\}$ with $K_i \cap C_v \cap \beta_M \neq \emptyset$,
some vertex of $K_i \cap C_v \cap \beta_M$ has a neighbor in $D$.}

Let $D$ be a component of $\beta_M \setminus N[v]$ with neighbors in as many cliques $K_i$ as possible. Let $K_1, \hdots, K_j$ be the cliques with a neighbor in $D$. By \eqref{path-helper}, there is a path $Q$ from $K_1$ to $K_{j+1}$ with
interior in $\beta_M \setminus N[v]$. Let $D'$ be the component of $\beta_M \setminus N[v]$ containing
$Q^*$, so $D \neq D'$. By the maximality of D, we may assume $D'$ does not have a neighbor in $K_2$.
By \eqref{path-helper}, there is also a path $R$ from $K_{j+1}$ to $K_2$ with interior in $\beta_M \setminus N[v]$.
Let $D''$ be the component of $\beta_M \setminus N[v]$ with $R^* \subseteq D''$.
Then $D'' \neq D,D'$. Let $P$ be a path from $K_1$ to $K_2$ with interior in $D$. Since $P, Q$, and $R$ are in distinct components of $\beta_M \setminus N[v]$, it follows that $P \cup Q \cup R$ is a hole of $\beta_M$. 
Now, $(P+Q+R,v)$ is a wheel in $\beta_M$, contrary to Lemma \ref{lemma:no-wheels-in-beta}. This proves \eqref{comp-attaches-everywhere}. \\

Let $I = \{1 \leq i \leq \ell : K_i \cap C_v \cap \beta_M \neq \emptyset\}$.
By \eqref{comp-attaches-everywhere}, there exists a component $D$ of $\beta_M \setminus N[v]$ such that $K_i$ has a neighbor in $D$ for every $i \in I$. By Lemma \ref{lemma:bounding-nbrhood-helper}, at most two of $K_1, \hdots, K_\ell$ have a neighbor in $D$. Therefore, $|I| \leq 2$. Recall that $N(v) \cap \beta_M \subseteq \tilde{C}_{v}$ and that $\tilde{C}_v = C_v \cup \left(\bigcup_{u \in N(v) \cap M} N(u) \cap N(v)\right)$. Since $G$ is diamond-free and $K_t$-free, it follows that $|N(v) \cap \beta_M| \leq 2t$.
\end{proof}

Now we prove the main result of this section. 
\begin{theorem}
Let $G \in \mathcal{C}_t^*$, let $w:V(G) \to [0, 1]$ be a weight function on $G$, and let $(\{v_1, \hdots, v_k\},$ $m, M, \tilde{S}_M)$ be the hub division of $G$. Let $\delta_t$ be the hub constant for $t$. Let $\beta_M$ be the central bag for $\tilde{S}_M$ and let $w_M$ be the inherited weight function on $\beta_M$. Assume that $\beta_M$ has a $(w_M, \frac{1}{2})$-balanced separator of size $C$ and that no vertex of $M$ is a pyramid apex in $\beta_M$. Then, $G$ has a $(w, \frac{1}{2})$-balanced separator of size $(2t + 4\delta_t)C$. 
\label{thm:extending-bs}
\end{theorem}
\begin{proof}
Let $X$ be a $(w_M, \frac{1}{2})$-balanced separator of $\beta_M$ with $|X| \leq C$. By Lemma \ref{lemma:grow-a-separator}, it follows that $X \cup (N[X \cap M] \cap \beta_M)$ is a $(w, \frac{1}{2})$-balanced separator of $G$. Let $u \in X \cap M$. By Lemma \ref{lemma:small-nbrs}, it follows that $|N_{\beta_M}(u) \setminus \Hub(\beta_M)| \leq 2t$. By Lemma \ref{lemma:no-wheels-in-beta}, $\Hub(\beta_M) \subseteq \{v_m, v_{m+1}, \hdots, v_k\}$, so it follows that $|N_{\beta_M}(u) \cap \Hub(\beta_M)| \leq 4\delta_t$. Therefore, $|X \cup (N[X \cap M] \cap \beta_M)| \leq (2t + 4\delta_t)|X|$. 
\end{proof}

Finally, we restate and prove Theorem \ref{thm:main-nonspecific}. 
\begin{theorem}
\label{thm:C-bdd-tw}
Let $G \in \mathcal{C}_t$. Then, $\tw(G) \leq (4t + 8\delta_t) \cdot \max(R(t, 4) + 1, 6t + 4\delta_t)$. 
\end{theorem}
\begin{proof}
By Lemma \ref{lemma:clique-cutsets-tw}, we may assume that $G$ has no clique cutset, so $G \in \mathcal{C}_t^*$. Let $w:V(G) \to [0, 1]$ be a weight function on $G$ with $w(G) = 1$, and let $(\{v_1, \hdots, v_k\}, m, M, \tilde{S}_M)$ be the hub division of $G$. Let $\beta_M$ be the central bag for $\tilde{S}_M$ and let $w_M$ be the inherited weight function on $\beta_M$. By Theorem \ref{thm:mainthm-betabs}, $\beta_M$ has a $(w_M, \frac{1}{2})$-balanced separator of size at most $\max(R(t, 4) + 1, 6t + 4\delta_t)$. Now, by Theorem \ref{thm:extending-bs}, $G$ has a $(w, \frac{1}{2})$-balanced separator of size $(2t + 4\delta_t) \cdot \max(R(t, 4) + 1, 6t + 4\delta_t)$. Finally, by Lemma \ref{lemma:bs-to-tw}, $\tw(G) \leq (4t + 8\delta_t) \cdot \max(R(t, 4) + 1, 6t + 4\delta_t)$. 
\end{proof}

\end{document}